\documentclass[12pt]{article}
\usepackage{amsmath,amssymb,amsfonts,amscd,amsxtra,amsthm}
\usepackage{latexsym}
\usepackage{bm}
\usepackage{natbib}

\input{epsf.sty}
\usepackage{epsfig}
\usepackage{graphicx,color}
\graphicspath{{./fig/}}

\usepackage{appendix}
\usepackage{wrapfig, blindtext}
\usepackage[top=1in, bottom=1in, left=1in, right=1in]{geometry}
\usepackage{hyperref}
\renewcommand{\theequation}{\thesection.\arabic{equation}}

\definecolor{darkblue}{rgb}{0.3,0,0.7}
\hypersetup{colorlinks,breaklinks,
	linkcolor=darkblue,urlcolor=darkblue,
	anchorcolor=darkblue,citecolor=darkblue}

\newtheorem{theorem}{Theorem}[section]

\newtheorem{definition}[theorem]{Definition}

\newtheorem{prop}[theorem]{Proposition}
\newtheorem{corollary}[theorem]{Corollary}

\newtheorem{remark}[theorem]{Remark}

\numberwithin{equation}{section}

\usepackage{subcaption}

\renewcommand\appendix{\par
  \setcounter{section}{0}
  \setcounter{subsection}{0}
  \setcounter{figure}{0}
  \setcounter{table}{0}
  \renewcommand\thesection{Appendix \Alph{section}}
  \renewcommand\theequation{\Alph{section}.\arabic{equation}}
  \renewcommand\thefigure{\Alph{section}.\arabic{figure}}
  \renewcommand\thetable{\Alph{section}.\arabic{table}}
  \renewcommand\thethm{\Alph{section}.\arabic{thm}}
}

\newcommand{\veps}{\varepsilon}

\newcommand{\ptl}{\partial}

\newcommand{\br}{\mathbb{R}}

\newcommand{\ii}{\mathrm{i}}

\newcommand{\bx}{\mathbf{x}}
\newcommand{\by}{\mathbf{y}}

\newcommand{\vt}{\hat{\mathbf{t}}}
\newcommand{\vn}{\hat{\mathbf{n}}}

\numberwithin{equation}{section}

\date{}

\title{Traveling edge states in massive Dirac equations along slowly varying edges}

\author{
Pipi Hu\thanks{\footnotesize Yau Mathematical Sciences Center, Tsinghua University, Beijing 100084 and Yanqi Lake Beijing Institute of Mathematical Sciences and Applications, Beijing 101408, China (\href{mailto:hpp1681@gmail.com}{hpp1681@gmail.com}).} 
\qquad Peng Xie\thanks{\footnotesize 
Department of Mathematics, The Hong Kong University of Science and Technology, Clear Water Bay, Kowloon, Hong Kong SAR (\href{mailto:mapengxie@ust.hk}{mapengxie@ust.hk}).}
\qquad Yi Zhu\thanks{\footnotesize 
Yau Mathematical Sciences Center, Tsinghua University, Beijing 100084 and Yanqi Lake Beijing Institute of Mathematical Sciences and Applications, Beijing 101408, China (\href{mailto:yizhu@tsinghua.edu.cn}{yizhu@tsinghua.edu.cn}).}}

\begin{document}
\maketitle

\begin{abstract}
Topologically protected wave motion has attracted considerable interest due to its novel properties and potential applications in many different fields. In this work, we study edge modes and traveling edge states via the linear Dirac equations with so-called domain wall masses. The unidirectional edge state provides a heuristic approach to more general traveling edge states through the localized behavior along slowly varying edges. We show the leading asymptotic solutions of two typical edge states that follow the circular and curved edges with small curvature by analytic and quantitative arguments.
\end{abstract}


\setcounter{equation}{0}
\setlength{\arraycolsep}{0.2em}

\section{Introduction}

The topological wave phenomena have sparked an explosion of the  interface features between distinct topological insulators \cite{ablowitz2013localized,drouot2020edge,fefferman2016bifurcations,fefferman2016edge,hasan2010topological}. One striking character of the so-called edge states is the existence of chiral propagating waves which are immune to the local defects in the sense of waves retaining on the edge robustly. This immunity is a delicate property in applied perspectives and it can be contributed to interpreting many ubiquitous physical scenarios. These studies are not only investigated by the electronic waves in condensed matter physics but also rapidly extended to photonics, water waves and related subjects \cite{delplace2017topological,fleury2016floquet,graf2021topology,lu2014topological,mousavi2015topologically,susstrunk2015observation,witten2016three,wu2018topological}.

In current work, we consider the dynamics of the edge state described by the two-component Dirac equation with a varying mass in the following canonical form:
\begin{equation}\label{eq.simple}
	\ii\partial_t\begin{pmatrix}
		\alpha_1 \\
		\alpha_2 \\
	\end{pmatrix}+\begin{pmatrix}
		\ii\partial_{x_2} & m(\bx)-\partial_{x_1}\\
		m(\bx)+\partial_{x_1}& -\ii\partial_{x_2}
	\end{pmatrix}
	\begin{pmatrix}
		\alpha_1 \\
		\alpha_2 \\
	\end{pmatrix}=0.
\end{equation}
where $\alpha_j=\alpha_j(t, \bx),~ j=1,2$ are complex-valued wave functions, $m(\bx) \in C(\mathbb{R}^2, \mathbb{R})$ is the mass term, and $``\mathrm{i}"$ is the imaginary unit. One can directly verify that the total energy $\mathcal{E}(t)=\int_{\mathbb{R}^2}|\alpha_1(t, \bx)|^2+|\alpha_2(t, \bx)|^2 d\bx$ is conserved since the Dirac operator behaves as a Hamiltonian. It also admits the global existence for the smooth solution with a smooth initial condition.

In homogeneous honeycomb latticed materials, Dirac points regularly appear at the spectrum band structure with the corresponding quasi-periodic eigenmodes and the wave packets around this degenerated point are dominated by the massless Dirac equation \cite{ablowitz2009conical,ablowitz2012nonlinear,fefferman2012honeycomb,fefferman2014wave,geim2007The,lee2019elliptic,neto2009electronic,novoselov2005two,raghu2008analogs}. However, this conical intersection disappears if time-reversal symmetry is broken in the material and then a local band gap emerges in the essential spectrum which leads to the insulating bulk \cite{fefferman2012honeycomb,haldane2008Possible,hasan2010topological,rechtsman2013Photonic}. The Dirac equation with a varying mass \eqref{eq.simple} arises from the effective envelopes of wave propagation in topological materials. Here, the mass $m(\bx)$ determines the distinct topology such that two adherent materials are topological insulators in bulks and the current or electromagnetic wave is permitted to travel along the contact edge \cite{bernevig2013topological,hu2020linear,raghu2008analogs,xie2019wave}. The associated edge, null domain of $m(\bx)$, separates the two dimensional materials with different topology in each part. Moreover, this novel electric conductivity elucidates the chirality and one unidirectional localized current flows along the edge only. Recently, the spectrum structure in honeycomb latticed medium also fascinates lots of attention from mathematical viewpoints. A variety of rigorous research has studied the existence of Dirac points and the local band gap brought after a time reversal symmetry breaking perturbation in domain wall and tight binding models \cite{fefferman2017topologically,fefferman2012honeycomb,keller2018spectral,lee2019elliptic,xie2021wave}. Meanwhile, one dimensional topologically protected edge state always occurs at the band gap when two adjacent medium state the distinct topological invariant associated with the Zak phase and Chern number \cite{ammari2020robust,ammari2020topologically,bal2019topological,drouot2020edge,fefferman2016edge,guo2019bloch,lee2019elliptic,lin2021mathematical}. Instead of dealing with the highly oscillated interface mode directly, a canonical way is to exploit the essentially homogenized envelope emerged by the time-harmonic massive Dirac equation which inherits the topological protected properties more clearly and intuitively. Studies about the existence of edge states or the derivation of governed envelopes---Dirac equations are carried out in many settings, such as microlocal analysis, transfer matrix method, Fredholm operator index, K-theory, and so on \cite{bal2019continuous,bal2019topological,drouot2020edge,lin2021mathematical,thiang2020edge}.


In physical applications, edge modes would also travel along various shapes of the interface where bulk defects happen \cite{bandres2016topological,cheng2016robust,ma2015guiding}. These physical phenomena stimulate the interests of wave propagation along the nontrivial edges. Recently, a class of Dirac equations with a small semi-classical parameter described the wave packets which propagate along the curved edge for long times \cite{bal2022magnetic,bal2021edge}. However, these effective models depend on the small parameter in the semi-classical equation and the curvature of nearly straight edge provides the limited effect to the time validity. In the current study, we will introduce an domain wall mass term and directly elucidate the classical dynamics of edge states when the interface curvature is very small. We seek the quasi-traveling edge states propagation pinned on the curved edge from the idea of modes along the straight interface and we also exploit a delicate modulation so that the accuracy of energy estimate will be improved.

The crucial result in our development of the quasi-traveling edge state is guided by the slowly varying edge perfectly. The effectiveness has distinct linear corrections with the edge curvature square. It can be carried out through the constructively asymptotic solution which is locally raised from the case of straight mass edge. We employ a well-prepared initial condition and then show a more accurate validity of this setup via the analytic and quantitative studies. From this scenario, it enlightens that the edge states governed by the macroscopic massive Dirac equation are topologically protected. The main results of the present work are summarized here:
\begin{itemize}\setlength{\itemsep}{0.5em}
	\item For the basic straight line edge, we employ the plane wave separation to discuss the unidirectional traveling edge state that is pinned on the edge in Proposition \ref{prop.move} and Corollary \ref{corollary:travel};
	\item For the slowly varying edge, we establish two typical quasi-traveling edge states along the circle with large radius and generic curves with the small curvature in Theorem \ref{thm1} and Theorem \ref{thm.simple} respectively. We demonstrate the reliability of herein developed asymptotic solutions and the residual errors only depend on time linearly and curvature quadratically.
\end{itemize}

The rest of this article is organized as follows. We discuss the solution to the traveling edge state described by massive Dirac equation with a unidirectional edge in Section \ref{sec3}. In Section \ref{sec4} and \ref{sec5}, two typical models with general domain wall masses show that the longtime stable asymptotic solutions are heuristically solved by regarding the partial edge as a local straight line along the tangent direction.

\section{Traveling edge states with the linear mass}\label{sec3}

In the physical setup, the ``edge" (or interface) comes from the connected boundary between two topological materials, and can be described by a smooth function \cite{drouot2020edge,lee2019elliptic}. For this sake, we define the domain wall mass term below.
\begin{definition}\label{def.edge}
	The mass term $m(\bx)$ is called the domain wall function if $m(\bx)$ is piecewise continuous, and can be written as $m(\bx)=\tilde{m}(u)$, $u=f(\bx)$ which satisfies:
	\begin{enumerate}
		\item $\tilde{m}(\cdot)\in L^\infty(\mathbb{R})$ is a monotonic transition function, $\tilde{m}(0)=0$, and the level set $\Gamma=\{\bx\in\br^2:f(\bx)=0\}=\tilde{m}^{-1}(0)$ represents a smooth curve;
		\item $\lim\limits_{u\to\pm \infty}\tilde{m}(u)=\pm m_\infty, \; m_\infty>0$ and $m(\bx)$ approaches to $\pm m_{\infty}$ rapidly away from $\Gamma$.
	\end{enumerate}
\end{definition}

Here, we illustrate two typical examples of domain wall functions from the literature \cite{fefferman2017topologically,fefferman2012honeycomb,lee2019elliptic} as follows:
\begin{equation*}
	\tilde{m}(u) = \tanh(u);\quad\quad \tilde{m}(u) =
	\begin{cases}
		-1, & \text{for } u < 0, \\
		0, & \text{for } u = 0, \\
		1, & \text{for } u > 0. \\
	\end{cases}
\end{equation*}
Henceforth we will drop the tilde and study a specified domain wall term for simplicity, i.e., the mass term is denoted by $m(f(\bx))$ with the edge $\Gamma=\{\bx\in\br^2:f(\bx)=0\}$.


We reveal the edge state with a linear edge curve $\Gamma$ and the Dirac equation has an explicit traveling wave solution. To this end, we assume that $m(f(\bx))=m(\vn\cdot\bx)$ where $\vn$ is a unit normal vector to the edge and $m(u)$ rapidly approaches to $\pm m_{\infty}$ as $u=\vn\cdot\bx\rightarrow \pm\infty$. Figure \ref{fig:sa} shows a vertical straight edge and Figure \ref{fig:sb} is a general linear case. We rewrite \eqref{eq.simple} under the linear edge form:
\begin{equation}\label{eq.line}
	\ii\partial_t
	\begin{pmatrix}
		\alpha_1 \\
		\alpha_2 \\
	\end{pmatrix}+\begin{pmatrix}
		\ii\partial_{x_2} & m(\vn\cdot\bx)-\partial_{x_1}\\
		m(\vn\cdot\bx)+\partial_{x_1}& -\ii\partial_{x_2}
	\end{pmatrix}
	\begin{pmatrix}
		\alpha_1 \\
		\alpha_2 \\
	\end{pmatrix}=0.
\end{equation}
Further, the tangent vector $\vt$ is obtained by rotating $\vn$ 90-degree counterclockwise. We introduce the azimuth $\theta\in[0,2\pi)$ to $x-$axis, and then $\vn$, $\vt$ can be represented as follow,
\begin{equation}\label{eq.cor}
	\vn = \begin{pmatrix}
		\cos\theta\\
		\sin\theta\\
	\end{pmatrix},\quad \quad
	\vt = \begin{pmatrix}
		-\sin\theta\\
		\cos\theta\\
	\end{pmatrix}.
\end{equation}

\begin{figure}[htbp]
	\centering
	\begin{subfigure}{0.3\textwidth}
		\includegraphics[width=\textwidth]{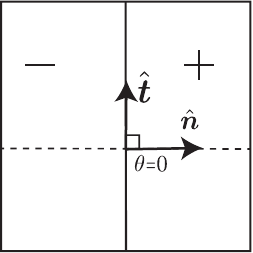}
		\caption{}\label{fig:sa}
	\end{subfigure}
	\begin{subfigure}{0.3\textwidth}
		\includegraphics[width=\textwidth]{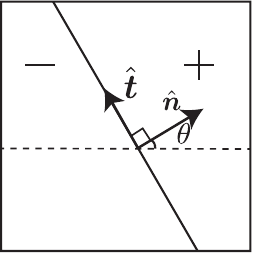}
		\caption{}\label{fig:sb}
	\end{subfigure}
	\begin{subfigure}{0.3\textwidth}
		\includegraphics[width=\textwidth]{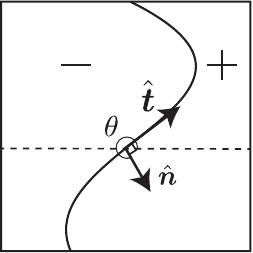}
		\caption{}\label{fig:sc}
	\end{subfigure}
	\caption{Left panel (A): the vertical straight line edge. Middle panel (B): the general case of straight line edge. Right panel (C): the general case of curved edge.}\label{fig.scheme}
\end{figure}

We would reveal the edge state with the linear mass and the Dirac equation \eqref{eq.line} has the plane wave solution. The result is summarized below.
\begin{prop}\label{prop.move}
	Suppose that the edge curve is a straight line with the unit normal vector $\vn$ and tangent vector $\vt$. For any parameter $k\in\mathbb{R}$, the Dirac equation \eqref{eq.line} has a plane wave solution as follows:
	\begin{equation}\label{eq.plane}
		\begin{pmatrix}
			\alpha_1\\
			\alpha_2\\
		\end{pmatrix} = \chi(\vn\cdot\bx)e^{\ii k(\vt\cdot\bx- t)}
		\begin{pmatrix}
			\cos\frac{\theta}{2}\\[0.3em]
			\ii\sin\frac{\theta}{2}\\
		\end{pmatrix},
	\end{equation}
	where $\chi(\cdot)=Ce^{-\int_0^{\cdot} m(s) ds}$ is a localized real-valued function with $C>0$ being the normalization constant.
\end{prop}

\begin{remark}
	Here, we can figure out that the traveling edge states would always propagate towards the positive direction of $\vt$ and decay along the $\vn$ direction which is the implication of the topological chirality. Before demonstrating this judgment, we would point out that $\binom{\chi(u)}{0}$ is an eigenfunction corresponding to the eigenvalue $-k$ of the one-dimensional (1D) Dirac operator
	\begin{equation}\label{op.dirac1d}
		\mathcal{D}_{k}=
		\begin{pmatrix}
			-k & m(u)-\partial_{u} \\
			m(u)+\partial_{u} & k \\
		\end{pmatrix}.
	\end{equation}
\end{remark}

\begin{proof}

We firstly perform the coordinate transformation,
	\begin{equation}\label{eq.transuv}
		u = \vn\cdot\bx,\quad v = \vt\cdot\bx,
	\end{equation}
	or equivalently by \eqref{eq.cor}
	\begin{equation*}
		\begin{pmatrix}
			u\\
			v\\
		\end{pmatrix} = \begin{pmatrix}
			\cos\theta & \sin\theta \\
			-\sin\theta & \cos\theta\\
		\end{pmatrix} \begin{pmatrix}
			x_1\\
			x_2\\
		\end{pmatrix}.
	\end{equation*}
	Let $\widetilde\alpha_j(t,u,v)=\alpha_j\big(t,\bx(u,v)\big),\, j=1,2$. The Dirac equation \eqref{eq.line} changes into
	\begin{equation}\label{eq.theta}
		\ii\partial_t\begin{pmatrix}
			\widetilde\alpha_1\\
			\widetilde\alpha_2\\
		\end{pmatrix}+
		\widetilde{\mathcal{D}}\begin{pmatrix}
			\widetilde\alpha_1\\
			\widetilde\alpha_2\\
		\end{pmatrix}=0.
	\end{equation}
	Here, the Dirac operator $\widetilde{\mathcal{D}}$ under the new coordinate is in the form of
	\begin{equation*}
		\widetilde{\mathcal{D}}=\begin{pmatrix}
			\ii(\sin\theta\partial_u+\cos\theta\partial_v) & m(u)-(\cos\theta\partial_u-\sin\theta\partial_v)\\
			m(u)+(\cos\theta\partial_u-\sin\theta\partial_v) &-\ii(\sin\theta\partial_u+\cos\theta\partial_v)
		\end{pmatrix}.
	\end{equation*}
	However, $\widetilde{\mathcal{D}}$ looks obscure by messing with $\theta$. To make it clear, we introduce a rotation transformation such that $\psi_j=\psi_j(t,u,v),~j=1,2$ satisfy
	\begin{equation}\label{eq.trans}
		\begin{pmatrix}
			\psi_1\\
			\psi_2\\
		\end{pmatrix}=S\begin{pmatrix}
			\widetilde\alpha_1\\
			\widetilde\alpha_2\\
		\end{pmatrix} \quad \text{and}~\quad 
		S=\begin{pmatrix}
			\cos\frac{\theta}{2} & -\ii\sin\frac{\theta}{2}\\[0.3em]
			-\ii\sin\frac{\theta}{2} & \cos\frac{\theta}{2}\\
		\end{pmatrix}.
	\end{equation}
	Here $S^*S=I$ and the asterisk $*$ indicates the conjugate transpose. Then, a direct calculation yields \eqref{eq.theta} into
	\begin{equation}\label{eq.stand}
		\ii\partial_t\begin{pmatrix}
			\psi_1 \\
			\psi_2 \\
		\end{pmatrix}+\begin{pmatrix}
			\ii\partial_{v} & m(u)-\partial_{u}\\
			m(u)+\partial_{u}& -\ii\partial_{v}\\\end{pmatrix}
		\begin{pmatrix}
			\psi_1 \\
			\psi_2 \\
		\end{pmatrix}=0,
	\end{equation}
	which is parallel to the standard form \eqref{eq.line} when $\theta=0$ shown in Figure \ref{fig:sa}.
	
	Substituting $\psi_j(t,u,v)=\widetilde\psi_j(u)e^{\ii k v+ \ii\omega t}$, $j=1,2$ into the above equation \eqref{eq.stand} develops an eigenvalue problem to the 1D Dirac operator,
	\begin{equation}\label{eq.eigen}
		\mathcal{D}_{k}\widetilde{\boldsymbol{\psi}}=\omega\widetilde{\boldsymbol{\psi}},\quad\quad \widetilde{\boldsymbol{\psi}}=\begin{pmatrix}
			\widetilde\psi_1 \\
			\widetilde\psi_2 \\
		\end{pmatrix},
	\end{equation}
	where the Dirac operator $\mathcal{D}_{k}$ is defined in Proposition \ref{prop.move}. Namely, we have
	\begin{align*}
		\ptl_u\widetilde{\boldsymbol{\psi}}=m(u)
		\begin{pmatrix}
			-1~ & 0 \\
			0 & 1
		\end{pmatrix}\widetilde{\boldsymbol{\psi}}+
		\begin{pmatrix}
			0 & -(k-\omega) \\
			-(k+\omega) & 0
		\end{pmatrix}\widetilde{\boldsymbol{\psi}}.
	\end{align*}
	From the right-hand side, the eigenvalues of the second matrix are $\pm\sqrt{k^2-\omega^2}$ and the corresponding eigenvectors are
	\begin{equation*}
		\boldsymbol{v}_1=
		\begin{pmatrix}
			k-\omega +\sqrt{k^2-\omega^2}\\
			-\sqrt{k^2-\omega^2}-(k+\omega)
		\end{pmatrix},\quad
		\boldsymbol{v}_2=
		\begin{pmatrix}
			k-\omega+\sqrt{k^2-\omega^2} \\
			\sqrt{k^2-\omega^2}+k+\omega
		\end{pmatrix}.
	\end{equation*}

	Therefore, $\widetilde{\boldsymbol{\psi}}$ can be written as the composition of $\chi_1(u)\boldsymbol{v}_1+\chi_2(u)\boldsymbol{v}_2$. However, the localized solution only exists when $\omega=-k$ and
	\begin{equation}\label{chi}
		\widetilde{\boldsymbol{\psi}}=\begin{pmatrix}
			\chi(u) \\ 0
		\end{pmatrix},\quad \chi(u)=C e^{-\int_0^u m(\tau) d\tau}.
	\end{equation}
	Here $C$ is the normalized coefficient. Some relevant results can also be found in \cite{lee2019elliptic, xie2019wave}.

\begin{remark}
	
	Specifically, if $m(\bx)=\tanh x_1$, the dispersion relationship and the localized eigenfunctions are displayed in Figure \ref{fig.eigen_function}.

	\begin{figure}[h]
		\centering%
		\subcaptionbox{Eigenvalue of the Dirac operator $\mathcal{D}_{k}$, where the straight line $\omega=-k$ results in the localized mode and the shadow area contributes to the oscillation modes.\label{figh:a}}[6cm]
		{\includegraphics[width=0.57\textwidth]{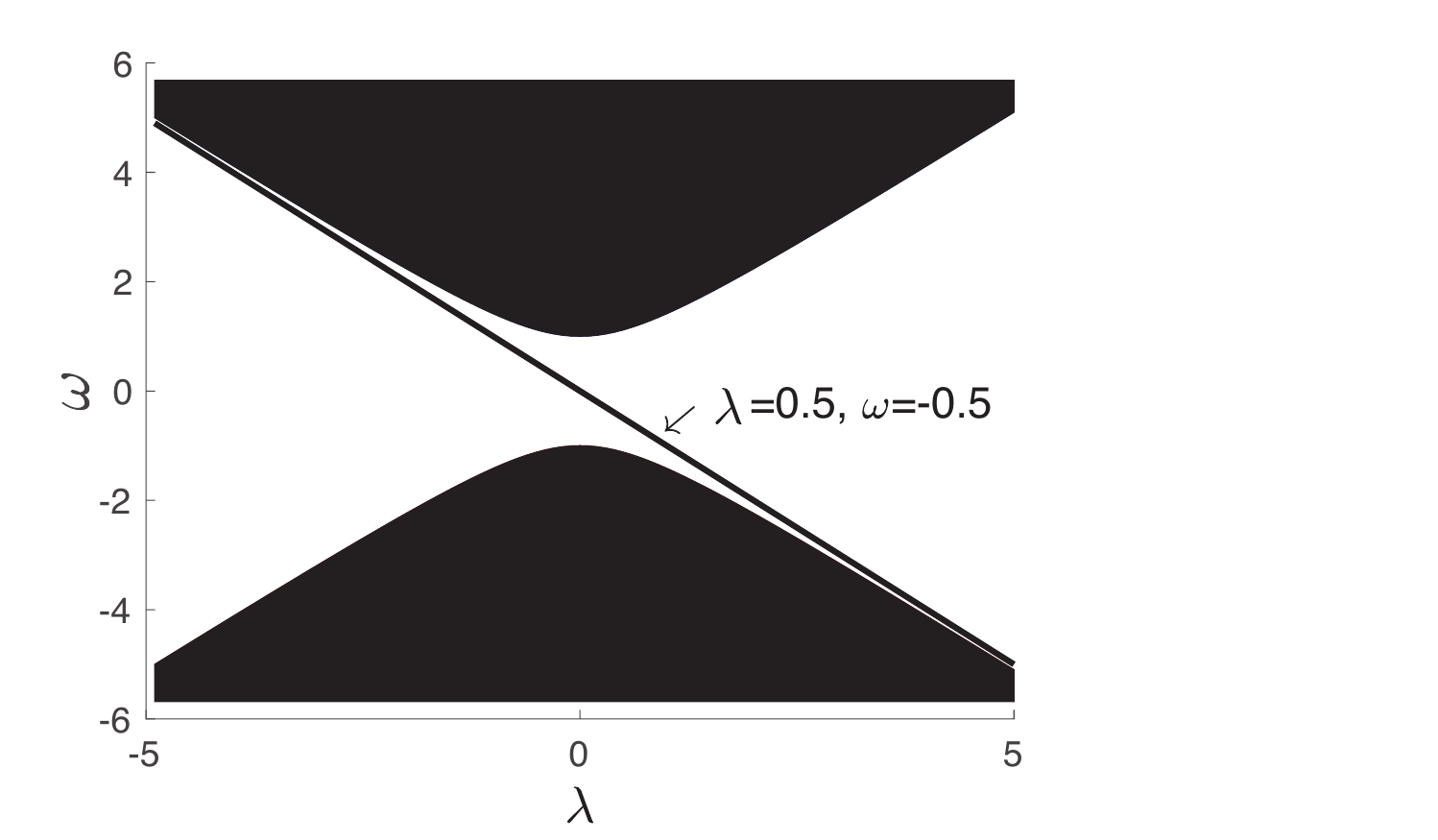}}%
		\hspace{2em}%
		\subcaptionbox{Eigenfunction at $k=0.5$\label{figh:b}}
		{\includegraphics[width=0.42\textwidth]{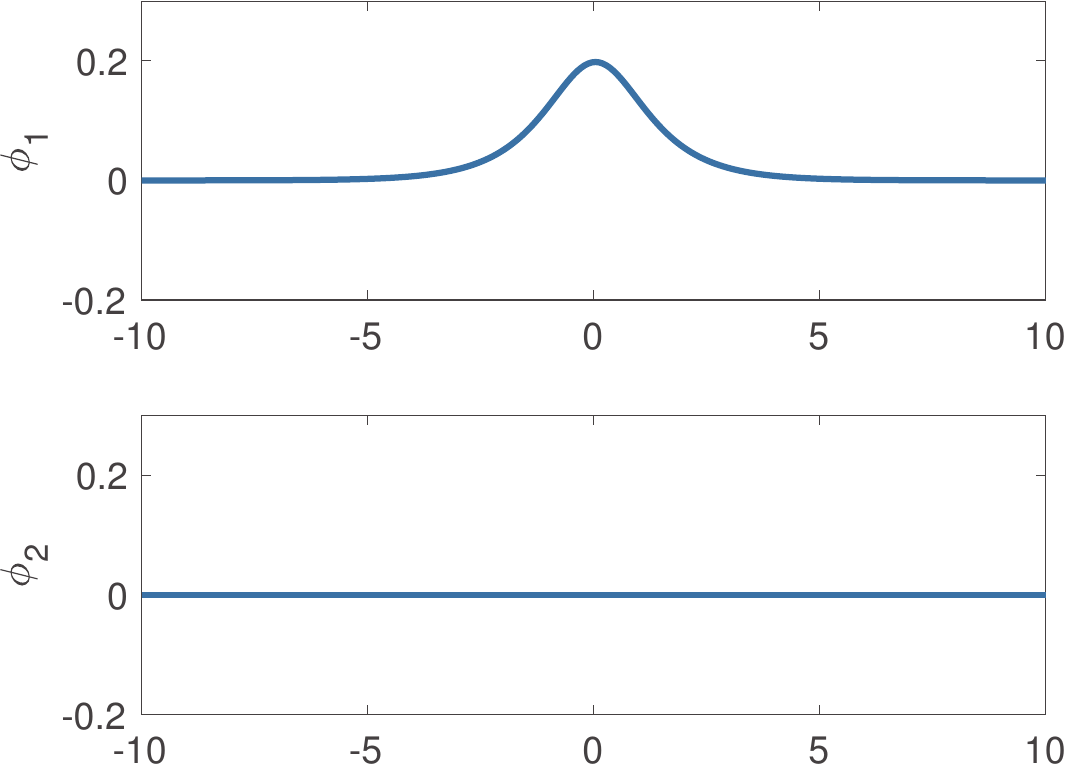}}
		\caption{The Edge Eigen for $m(\bx)=\tanh x_1$}
		\label{fig.eigen_function}
	\end{figure}
	
\end{remark}

We immediately obtain the plane wave solution of \eqref{eq.stand} as
\begin{equation}\label{eq.plane2}
	\begin{pmatrix}
		\widetilde{\psi}_1 \\
		\widetilde{\psi}_2 \\
	\end{pmatrix}=\chi(u) e^{\ii k (v- t)}
	\begin{pmatrix}
		1\\
		0\\
	\end{pmatrix} .
\end{equation}

Applying the inverse transform of \eqref{eq.trans} to \eqref{eq.plane2}, we eventually obtain the plane wave solution \eqref{eq.plane}. This completes the proof.
\end{proof}

The dispersion curve $\omega(k)$ corresponding to the localized eigenfunction is a straight line with the slope $-1$, which means the modes of the form \eqref{eq.plane} with different wavenumber $k$ has the same group velocity $v_{\text{group}}=-\frac{\partial \omega(k)}{\partial k}=1$ (note that the energy parameter $\omega$ selected here differs from the settings in physics by a negative sign). In other words, the Dirac system restricted to edge modes is dispersionless. A similar discussion as the above plane wave result carries out the localized solution in next corollary.

\begin{corollary}\label{corollary:travel}

Give the continuous function $g(\cdot)\in L^2(\br)$. If the initial input to the edge of problem \eqref{eq.line} is in the following localized form
\begin{equation*}
	\begin{pmatrix}
		\alpha_1(0, \bx)\\
		\alpha_2(0, \bx)\\
	\end{pmatrix} =\chi(\vn\cdot\bx)g(\vt\cdot\bx)
	\begin{pmatrix}
		\cos\frac{\theta}{2}\\[0.3em]
		\ii\sin\frac{\theta}{2}\\
	\end{pmatrix}.
\end{equation*}
Then, the localized traveling wave solution admits
\begin{equation}\label{eq.taveling}
	\begin{pmatrix}
		\alpha_1(t, \bx)\\
		\alpha_2(t, \bx)\\
	\end{pmatrix} = \chi(\vn\cdot\bx)g(\vt\cdot\bx-t)
	\begin{pmatrix}
		\cos\frac{\theta}{2}\\[0.3em]
		\ii\sin\frac{\theta}{2}\\
	\end{pmatrix}.
\end{equation}

\end{corollary}

This is a one way traveling wave along the edge with the velocity $v=1$. Till now, we have built the exact traveling localized waves along a straight edge through the separation of variables. However, it is quite involved to derive the traveling waves with a general mass. If $m(f(\bx))$ is a domain wall function with the edge curve $\Gamma$ which can be locally treated as a straight line.

To be more specific, we will propose two typical edge states where the edges can be locally treated as straight lines. We exploit the longtime stable asymptotic behaviors that cling to the edges locally other than solving the wave guidance derived from ODE \cite{bal2022magnetic,bal2021edge}. In the next two sections, one kind of edge is a circular ring with a sufficiently large radius, and the other is a slowly varying curve which is generated by adding a small perturbation to the straight line at the normal direction.

\section{Quasi-traveling edge state along the circle}\label{sec4}

Suppose that the mass term remains invariantly along the angular direction in polar coordinates, i.e., the edge curve is a circle. Namely, the mass term can be described by
\begin{equation}\label{eq.circle}
	m(f(\bx))=m(\sqrt{x_1^2+x_2^2}-R),
\end{equation}
where $R\in(0,+\infty)$ indicates the radius of a circle. Let the reference system alter into the polarization coordinates if $x_1=r\cos{\theta}$, $x_2=r\sin\theta$ with $r\in[0,+\infty)$, $\theta\in[0,2\pi)$.
Without loss of generality, we assume that $\exists~r_0>0$, $|m(u)|>\frac{~m_{\infty}}{2}$ provided $|u|>r_0$.

In such a case, the Dirac equation \eqref{eq.line} under the polarization coordinates admits the following form,
\begin{equation}\label{eq.circlecase}
	\small	\ii\partial_t
	\boldsymbol{\alpha}+\widetilde{\mathcal{D}}\boldsymbol{\alpha}=\boldsymbol{0},
\end{equation}
where $\boldsymbol{\alpha}=\binom{\alpha_1}{\alpha_2}$, $\alpha_j=\alpha_j(t,r,\theta),~j=1,2$, and the new Dirac operator is
\begin{equation}
	\widetilde{\mathcal{D}}=\begin{pmatrix}
		\ii\sin\theta\partial_r+\frac{\ii}{r}\cos\theta\partial_\theta & m(r-R)-\cos\theta\partial_r+\frac{1}{r}\sin\theta\partial_\theta\\[0.3em]
		m(r-R)+\cos\theta\partial_r-\frac{1}{r}\sin\theta\partial_\theta & -\ii\sin\theta\partial_r-\frac{\ii}{r}\cos\theta\partial_\theta \\
	\end{pmatrix}.
\end{equation}
Moreover, $e^{\ii\widetilde{\mathcal{D}}t}$ also represents a Dirac group and is unitary in $L^2(\br^2)$ for all $t>0$.

Observe that the traveling wave solution to the Dirac equation \eqref{eq.line} is of the form \eqref{eq.taveling}. With a circular edge, it is not surprising that we seek for the ansatz by treating the circle as a straight line locally, i.e.,
\begin{equation}\label{eq.asymptotic}
	\begin{pmatrix}
		\alpha_1(t,r,\theta) \\
		\alpha_2(t,r,\theta) \\
	\end{pmatrix}_R
	= \phi(r)g(\theta-\frac{t}{R})
	\begin{pmatrix}
		\cos\frac{\theta}{2} \\[0.3em]
		\ii\sin\frac{\theta}{2} \\
	\end{pmatrix}.
\end{equation}
Then, we need to derive the validity of the above approximation. Plugging the right hand side of \eqref{eq.asymptotic} into the equation \eqref{eq.circlecase} deduces the residual terms as
\begin{equation}\label{id:lhs}
	 \Bigg(\ii\partial_t
	 +\widetilde{\mathcal{D}}\Bigg)
	 \begin{pmatrix}
	 	\alpha_1 \\
	 	\alpha_2 \\
	 \end{pmatrix}_R= 
	 \begin{pmatrix}
		\ii (\phi'+m\phi+\frac{1}{2r}\phi)g\sin\frac{\theta}{2}+\ii(\frac{1}{r}-\frac{1}{R})\phi g'\cos\frac{\theta}{2} \\[0.3em]
		(\phi'+m\phi+\frac{1}{2r}\phi)g\cos\frac{\theta}{2}-(\frac{1}{r}-\frac{1}{R})\phi g'\sin\frac{\theta}{2}\\		
	\end{pmatrix}:=RHS.
\end{equation}
However, the present mass term \eqref{eq.circle} could not consistent with the domain wall precisely. Some extra constraints on $\phi(r)$ are needed as $r\rightarrow 0^{+}$, otherwise the $RHS$ may lead to singularities at the origin.

Now, we give a rigorous clarification for asymptotic solution to the edge state pinned on the circle.
\begin{theorem}\label{thm1}
	Suppose that the radius $R>\max\{2 r_0, \frac{1}{~m_{_\infty}}\}$ is large enough. The Dirac equation \eqref{eq.circlecase} is spacially defined under the polar coordinates with the circular edge incurred by \eqref{eq.circle}. Moreover, $g\in C^1 [0,2\pi]$ with $g(2\pi)=g(0)$ and there exist $0<r_1<r_2<\frac12$, $\phi(r)\in C^1[0,+\infty)$ satisfies:
	\begin{enumerate}
		\item if $r\in[0, r_1 R]$, $\phi(r)=0$;
		\item if $r\in(r_1 R, r_2 R)$, $0\leqslant \phi'(r)\leqslant \max\big\{\phi'(r_2 R), \frac{2}{(r_2-r_1)R}\phi(r_2 R)\big\}$;
		\item if $r\in [r_2 R,+\infty)$, $\phi(r)=\frac{1}{\sqrt{R}~}\chi(r-R)e^{-\frac{r-R}{2R}}$, where $\chi(\cdot)$ is defined in Proposition \ref{prop.move} and $\frac{1}{\sqrt{R}~}$ is the normalized parameter in $L^2(\br^2)$.
	\end{enumerate}
	Assume that the initial condition is perfectly matched. Then, for any $t> 0$, there is a constant $C>0$ independent of $t$, $R$ such that
	\begin{equation*}
		\Bigg\|
		\begin{pmatrix}
			\alpha_1(t,r,\theta) \\
			\alpha_2(t,r,\theta)
		\end{pmatrix}
		-\begin{pmatrix}
			\alpha_1(t,r,\theta) \\
			\alpha_2(t,r,\theta)
		\end{pmatrix}_R
		\Bigg\|_{L^2(\br^2)} \leqslant C t \frac{1}{R^2}.
	\end{equation*}
	
\end{theorem}

\begin{proof}
Let $\boldsymbol{\eta}=\boldsymbol{\eta}(t,r,\theta)$ indicate the error of the asymptotic solution \eqref{eq.asymptotic}, i.e.,
\begin{equation*}
	\boldsymbol{\eta}=
	\begin{pmatrix}
		\alpha_1(t,r,\theta) \\
		\alpha_2(t,r,\theta)
	\end{pmatrix}
	-\begin{pmatrix}
		\alpha_1(t,r,\theta) \\
		\alpha_2(t,r,\theta)
	\end{pmatrix}_R.
\end{equation*}
Owing to $RHS$ given in \eqref{id:lhs}, the error $\boldsymbol{\eta}$ evolves like:
\begin{equation*}
	\mathrm{i}\partial_t\boldsymbol{\eta}+\tilde{\mathcal{D}}\boldsymbol{\eta}=-RHS.
\end{equation*}
Then, it follows from the Duhamel's principle that
\begin{equation*}
	\boldsymbol{\eta}= \mathrm{i}\int_0^t e^{\mathrm{i}\tilde{\mathcal{D}}(t-s)}RHS~ ds,\quad \forall~ t> 0.
\end{equation*}
According to the fact that $e^{\mathrm{i}\tilde{\mathcal{D}}t}$ is unitary in $L^2(\br^2)$, we can directly obtain
\begin{equation}\label{eta:1}
	\|\boldsymbol{\eta}\|_{L^2(\br^2)}\leqslant \int_0^t \|RHS\|_{L^2(\br^2)}~ ds.
\end{equation}

As it has been stated before, if $0\leqslant r\leqslant r_1 R$, $\phi(r)\equiv 0$. When $r\in(r_1 R, r_2 R)$, for any $n\geqslant 0$, it implies
\begin{equation*}
	\big|\phi'+m\phi+\frac1{2r}\phi\big|\leqslant C \frac{1}{R^n},\quad \big|(\frac1r-\frac1R)\phi\big|\leqslant C \frac{1}{R^n}.
\end{equation*}
Here the constants $C$ are independent of $R$.

If $r\geqslant r_2 R$, we have
\begin{equation*}
	\phi'+m\phi=-\frac{1}{2R}\phi.
\end{equation*}
Then, for $r\in[r_2 R, R-r_0]\cup [R+r_0,+\infty)$, it follows that
\begin{equation*}
	\big|(\frac1r-\frac1R)\phi(r)\big|\leqslant C\frac{|R-r|}{rR^{\frac32}}e^{-\frac{m_{\infty}}{2}|R-r|}.
\end{equation*}
Noting the boundedness of $\chi(\cdot)$, we claim the result below as $r\in (R-r_0, R+r_0)$,
\begin{equation*}
	\big|(\frac1r-\frac1R)\phi(r)\big|\leqslant C\frac{1}{R^{\frac52}},
\end{equation*}
which also explains the reason behind the modulated factor $e^{-\frac{r}{2R}}$ into $\phi(r)$.

Consequently, for any $t> 0$, it turns out that
\begin{align*}
	\|RHS\|_{L^2(\br^2)}^2 \leqslant&~ C\int_0^{2\pi}\int_0^{\infty} \Big[\big|(\phi'+m\phi+\frac{1}{2r}\phi)g\big|^2 +\big|(\frac{1}{r}-\frac{1}{R})\phi g'\big|^2\Big] rdrd\theta \\
	\leqslant&~ C\int_{r_1 R}^{r_2 R} \Big[\big|\phi'+m\phi+\frac{1}{2r}\phi\big|^2 +\big|(\frac{1}{r}-\frac{1}{R})\phi\big|^2\Big] rdr  + C\int_{r_2 R}^{\infty}\big|(\frac{1}{r}-\frac{1}{R})\phi\big|^2 rdr \\
	\leqslant&~ C \frac1{R^{2n}} + C \frac{1}{R^4} \\
	\leqslant&~ C \frac{1}{R^4}.
\end{align*}
Here we choose $n\geqslant2$ to ensure the estimate consistently.

According to the formula in \eqref{eta:1}, for any $t>0$, we have
\begin{equation*}
	\|\boldsymbol{\eta}\|_{L^2(\br^2)} \leqslant C t\frac{1}{R^2}.
\end{equation*}
This completes the proof.

\end{proof}

Along the radius direction, the domain wall is negative inside the circle and positive outside. Therefore, the counterclockwise traveling direction obeys the chiral property, which leave the positive mass on the right.

It is apparent that $\phi(r)$ is not exactly the same as $\chi(r-R)$. For convenience, we give the comparison between $\sqrt{R}\phi(r)$ and $\chi(r-R)$ in Figure \ref{fig.mode_compare} which shows $\chi(r-R)$ is symmetric about $r=R$ but $\phi(r)$ is not. Here and in the next section, we numerically simulate the traveling waves to support our analysis by employing the pseudo-spectral method of fourth-order Runge-Kutta time integration \cite{bao2017numerical}. With different radii, we compute the asymptotic solution \eqref{eq.asymptotic} at the same time and it indicates that the errors go like $\mathcal{O}(\frac1{R^2})$ in $L^2(\br^2)$. To display the improvement of our ansatz more intuitively, we also numerically show the setup which has no $e^{-\frac{r-R}{2R}}$ in $\phi(r)$ with errors dropping to $\mathcal{O}(\frac1{R})$ as shown in Figure \ref{fig.maxfit}. In Figure \ref{fig.circle}, it carries out the numerical simulation patterns with the radius $R=40$ at four successive times. The waves travel around circle with negligible energy leaking into the bulk as $R>0$ large enough.

\begin{figure}[htbp]
	\centering
	\includegraphics[width=0.7\textwidth]{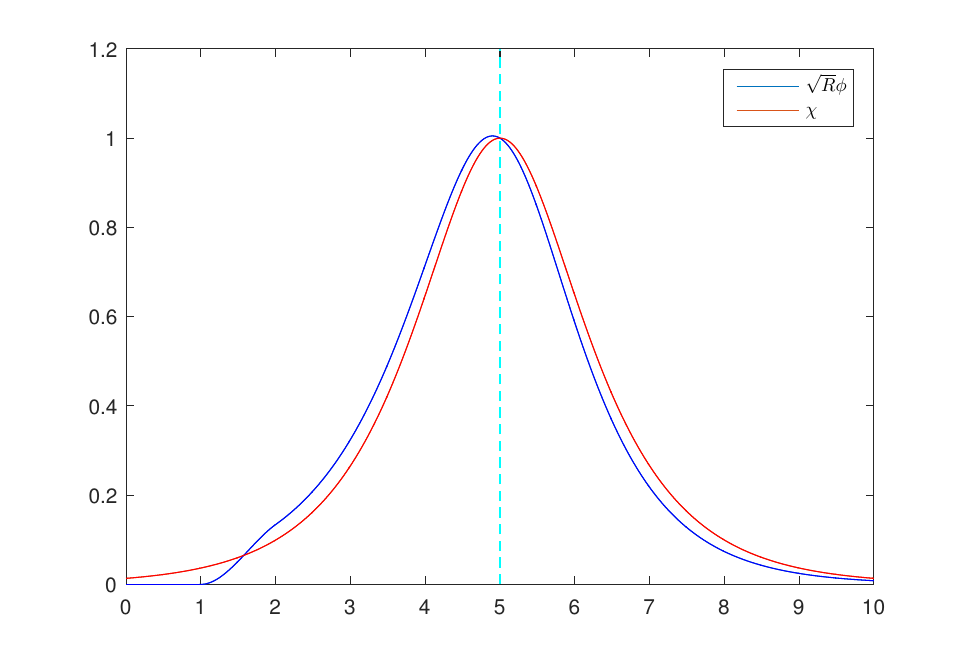}
	\caption{The graphs of quasi-circle mode $\sqrt{R}\phi(r)$ and line mode  $\chi(r-R)$ when $m(\cdot)=\tanh(\cdot)$ and $R=5$.}
	\label{fig.mode_compare}
\end{figure}

\begin{figure}[htbp]
	\centering
	\includegraphics[width=0.9\textwidth]{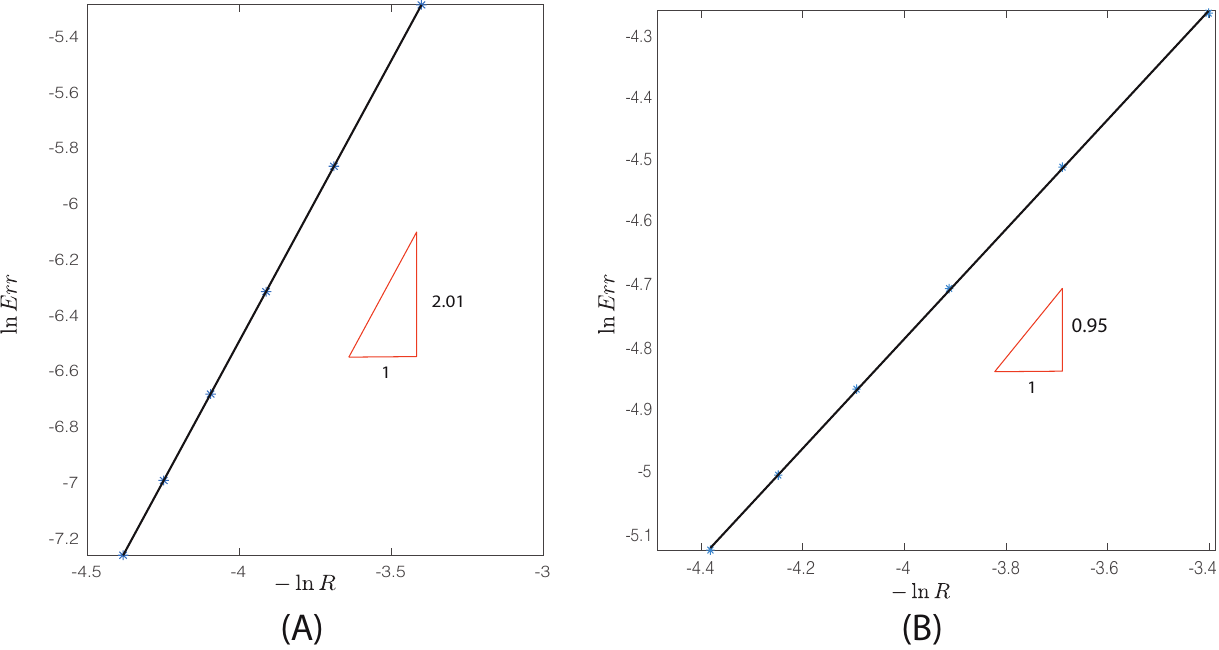}
	\caption{The $L^2$ errors for quasi-traveling edge states on circles at $t=5$.
	(A): The errors of \eqref{eq.asymptotic} is $\mathcal{O}(\frac{1}{R^2})$.
	(B): The errors of the case when $\phi(r)$ drops $e^{-\frac{r-R}{2R}}$ is $\mathcal{O}(\frac{1}{R})$.
} 
\label{fig.maxfit}
\end{figure}

\begin{figure}[htbp]
	\centering
	\includegraphics[width=0.88\textwidth]{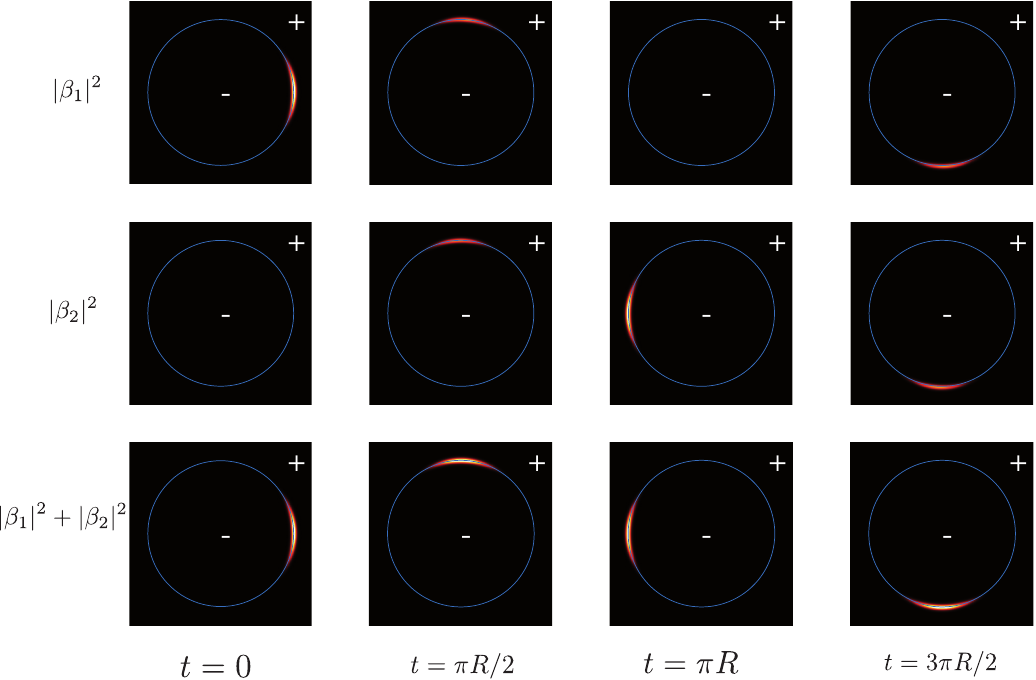}
	\caption{Quasi-traveling edge states with the circular radius $R=40$ at four successive time: $t = 0,~ \frac{\pi R}{2},~\pi R,~\frac{3\pi R}{2}$. The spots on circles from the top down are given by $|\alpha_1|^2$, $|\alpha_1|^2$ and $|\alpha_1|^2+|\alpha_2|^2$.}
	\label{fig.circle}
\end{figure}

\section{Quasi-traveling edge state along the smooth curve}\label{sec5}

In this section, we consider a family of more general edge states which highly propagate along the smooth edge curves (or interface). Observe that the straight line edge state with the rotation angle $\theta$ admits the propagating form of \eqref{eq.taveling}. In Figure \ref{fig:sc}, the smooth edge curve can be locally treated as going towards the tangent direction. It sheds some light on that the curved edge states may comply with the localized solution traveling along the tangent line of the curve. Hence, we study the quasi-traveling waves when the straight edge curve is slowly disturbed.

Assume that the edge curve is a small perturbation to the vertical line, i.e.,
\begin{equation}\label{curve2}
	\Gamma=\left\{\bx\in\br^2: x_1+h(\veps x_2)=0\right\}.	
\end{equation}
Here $0<\veps\ll1$ and $h(\veps x_2)$ indicates the small perturbation to the straight line. Other more general edge curves can be treated similarly by coordinate rotation. After that, the Dirac operator with the domain wall function $m\big(x_1+h(\veps x_2)\big)$ is in the form of
\begin{equation*}
	\mathcal{D}^{\veps}=\begin{pmatrix}
		\ii\partial_{x_2} & m\big(x_1+h(\veps x_2)\big)-\partial_{x_1}\\
		m\big(x_1+h(\veps x_2)\big)+\partial_{x_1}& -\ii\partial_{x_2}
	\end{pmatrix}.
\end{equation*}
Therefore, the Dirac equation \eqref{eq.line} alters into
\begin{equation}\label{eq.simpleasym}
	\ii\partial_t\begin{pmatrix}
		\alpha_1 \\
		\alpha_2 \\
	\end{pmatrix}+\mathcal{D}^{\veps}
	\begin{pmatrix}
		\alpha_1 \\
		\alpha_2 \\
	\end{pmatrix}=0.
\end{equation}
Under this setup, we move forward to establish the smooth edge states in an asymptotic way. Recalling the edge curve equation in \eqref{curve2} generates the unit normal and tangent vectors at each point on the curve:
\begin{align*}
	\vn&=\frac{1}{\sqrt{1+\veps^2 h'^2(\veps x_2)}~}\begin{pmatrix}
		1\\
		\veps h'(\veps x_2)
	\end{pmatrix}=\begin{pmatrix}
		1\\
		\veps h'(\veps x_2)
	\end{pmatrix}+\mathcal{O}(\veps^2),\\
	\vt&=\frac{1}{\sqrt{1+\veps^2 h'^2(\veps x_2)}~}\begin{pmatrix}
		-\veps h'(\veps x_2) \\
		1
	\end{pmatrix}=\begin{pmatrix}
		-\veps h'(\veps x_2) \\
		1
	\end{pmatrix}+\mathcal{O}(\veps^2).
\end{align*}

With the help of the straight line solutions \eqref{eq.taveling}, it is not surprising to develop an analogous result traveling along the slowly varying edge. From the curve function \eqref{curve2} and the above tangent vector $\vt$, it yields the asymptotic solution as follows:
\begin{equation}\label{solu:curve}
	\begin{pmatrix}
		\alpha_1 \\
		\alpha_2 \\
	\end{pmatrix}_{\veps}= \chi\bigl(x_1+h(\veps x_2)\bigr) g\bigl(-\veps h'(\veps x_2)(x_1+h(\veps x_2))+x_2-t\bigr)
	\begin{pmatrix}
		1 \\
		\frac{\ii}{2}\veps h'(\veps x_2) \\
	\end{pmatrix}.
\end{equation}
The validity of this construction can be demonstrated up to $\mathcal{O}(\veps^2)$ below.
\begin{theorem}\label{thm.simple}
	
	Let $0<\veps\ll 1$, $\chi(\cdot)$ be denoted as \eqref{chi} and  $g(\cdot)\in \mathcal{S}(\br)$. The domain wall function $m(\cdot)$ with the edge curve $\Gamma$ is defined by \eqref{curve2}. The edge perturbation $h(\cdot)\in C^2(\br)$ and $h'(\cdot),~h''(\cdot)$ are both bounded on $\br$. Suppose that the initial condition is perfectly matched, i.e., $\binom{\alpha_1}{\alpha_2}(0, \bx)=\binom{\alpha_1}{\alpha_2}_{\veps}(0, \bx)$. Then for any $t> 0$,
	\begin{equation*}
		\Bigg\|\begin{pmatrix}
			\alpha_1 \\
			\alpha_2 \\
		\end{pmatrix}(t, \bx)-\begin{pmatrix}
			\alpha_1 \\
			\alpha_2 \\
		\end{pmatrix}_{\veps}(t, \bx)\Bigg\|_{L^2(\mathbb{R}^2)}< C t\veps^2.
	\end{equation*}
	Here $C$ is a generic constant independent of $t$ and $\veps$.
\end{theorem}

\begin{proof}
	
Let the error $\boldsymbol{\eta}(t,\bx)=\binom{\alpha_1}{\alpha_2}(t,\bx)-\binom{\alpha_1}{\alpha_2}_{\veps}(t,\bx)$. Substituting the above formal solution \eqref{solu:curve} into \eqref{eq.simpleasym}, the evolution of $\boldsymbol{\eta}(t,\bx)$ arrives at
	\begin{equation*}
		\ii\partial_t\boldsymbol{\eta}+\mathcal{D}^{\veps}\boldsymbol{\eta} = -\frac12\veps^2
		\begin{pmatrix}
			-\ii 2(x_1+h)\chi g' h''-\ii \chi g'{h'}^2 \\
			\chi g h''+\chi'g{h'}^2-\veps (x_1+h) \chi g' h'h''-\veps \chi g'{h'}^3
		\end{pmatrix}.
	\end{equation*}
	By employing the same procedure in previous circular edge arguments, it suffices to estimate the above residuals.
	
	For convenience, we employ the coordinate transformation by letting $y_1=x_1+h(\veps x_2)$, $y_2=x_2$, and then the Jacobi determinant is identically equal to $1$. Assume that $|h'(\cdot)|\leqslant C_1$ uniformly on $\br$. For any $t> 0$, we firstly build the following estimate:
	\begin{align*}
		& \int_{\mathbb{R}^2}\Big|\chi(y_1) g\bigl(-\veps y_1 h'(\veps y_2)+y_2-t\bigr)\Big|^2 d\by \\
		= &~\int_{\mathbb{R}} \Big( \int_{|y_2-t|>2C_1\veps|y_1|}+\int_{|y_2-t|\leqslant 2C_1\veps|y_1|}\Big) \Big|\chi(y_1) g\bigl(-\veps y_1 h'(\veps y_2)+y_2-t\bigr)\Big|^2 d y_2 d y_1 \\
		\leqslant &~ C\int_{\mathbb{R}}\int_{|y_2-t|>2C_1\veps|y_1|}\chi^2(y_1) \frac{1}{1+|\frac{y_2-t}{2}|^{^2}} dy_2 dy_1+C \int_{\mathbb{R}}\int_{|y_2-t|\leqslant 2C_1\veps|y_1|} \chi^2(y_1) d y_2 d y_1 \\
		\leqslant &~C\int_{\mathbb{R}^2}\chi^2(y_1) \frac{1}{1+|\frac{y_2-t}{2}|^{^2}} d\by+C \int_{\mathbb{R}}\veps |y_1| \chi^2(y_1) d y_1 \\
		\leqslant &~C<+\infty.
	\end{align*}
	Here we use the rapidly decreasing property of $\chi(\cdot)$ and $g(\cdot)$. Hence, for any $t> 0$, $\binom{\alpha_1}{\alpha_2}_{\veps}(t,\bx)\in L^2(\br^2)^2$ and a similar strategy gives estimates to the residual,
	\begin{align*}
		&~\Bigg\|\begin{pmatrix}
			-\ii 2(x_1+h)\chi g' h''-\ii \chi g'{h'}^2 \\
			\chi g h''+\chi'g{h'}^2-\veps (x_1+h) \chi g' h'h''-\veps \chi g'{h'}^3
		\end{pmatrix}\Bigg\|_{L^2(\br^2)} \\
		=&~\Bigg\|\begin{pmatrix}
			-\ii 2y_1 \chi g'h''-\ii\chi g'{h'}^2 \\
			\chi g h''+\chi'g{h'}^2-\veps y_1 \chi g' h'h''-\veps \chi g'{h'}^3
		\end{pmatrix}\Bigg\|_{L^2(\br^2)}<+\infty.
	\end{align*}
	
	Thanks to the fact that $e^{\mathrm{i}\mathcal{D}^{\veps}t}$ is unitary in $L^2(\br^2)$, we can directly move forward to conclude that
	\begin{equation*}
		\|\boldsymbol{\eta}(t,\bx)\|_{L^2(\br^2)}< Ct\veps^2,\quad \forall~t>0,
	\end{equation*}
	where $C$ is a generic constant.
	
\end{proof}

\begin{figure}[htbp]
	\centering
	\includegraphics[width=0.46\textwidth]{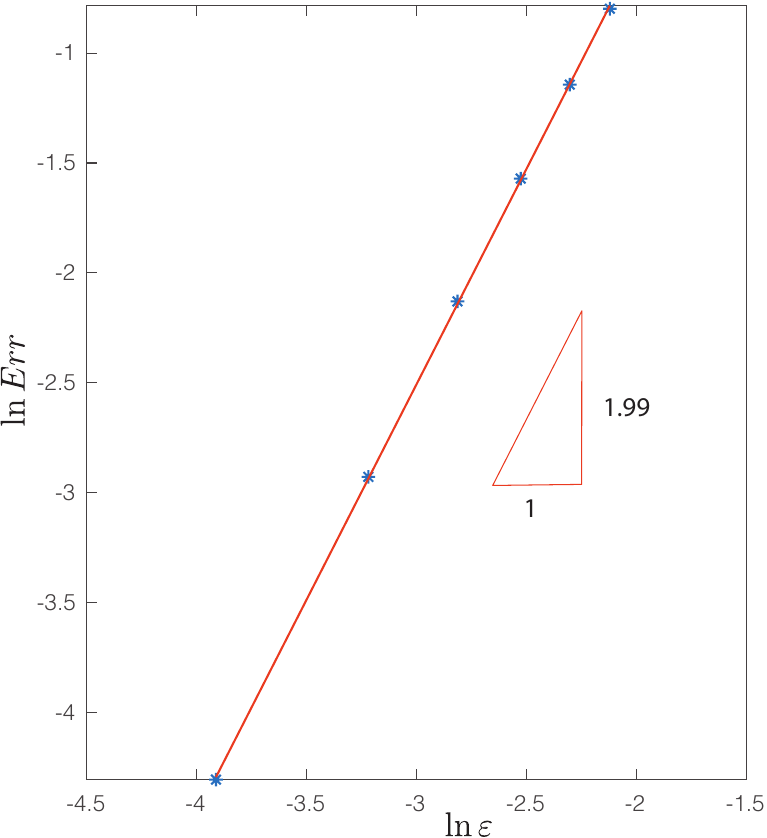}
	\caption{The $L^2$ error of the quasi-traveling edge state along the smooth curve is $\mathcal{O}(\veps^2)$ at $t=5$.}\label{fig.curveerr}
\end{figure}

\begin{figure}[htbp]
	\centering
	\includegraphics[width=0.86\textwidth]{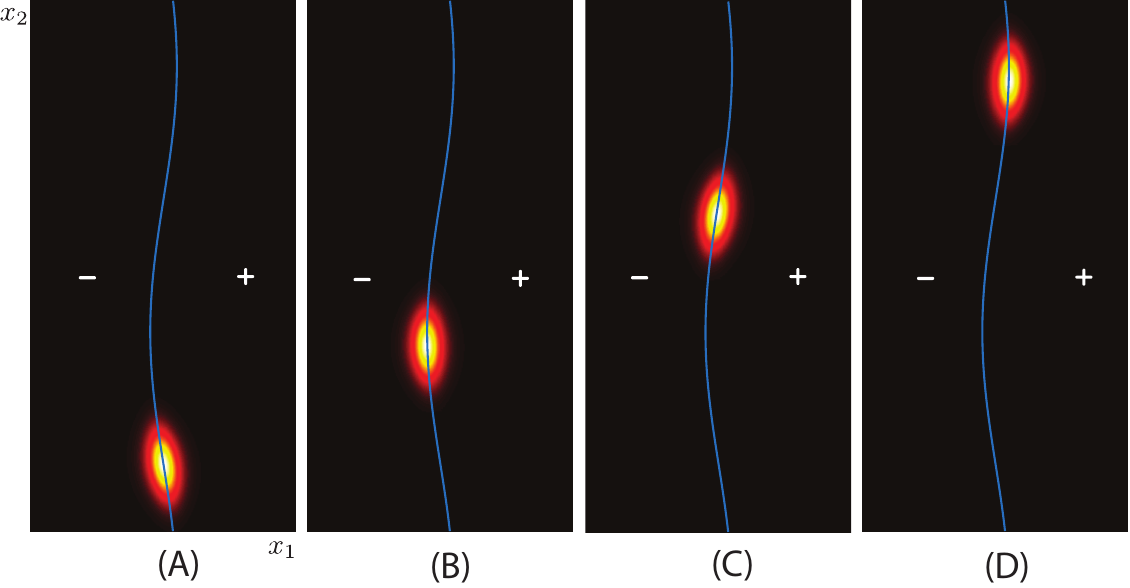}
	\caption{Let $\veps=0.2$, the snapshots show energies of traveling waves along the smooth-curved edge $x_1=-\sin(0.2x_2)$ in the domain $[-10,10]\times[-10\pi, 10\pi]$ at four successive time: $t=0,~10,~15,~20$.}\label{fig.sin}
\end{figure}

In Figure \ref{fig.curveerr} and \ref{fig.sin}, we also numerically show the error dependence on the curvature and simulating patterns when the edge curve is $x_1=-\sin(0.2 x_2)$. From the above theorem, we establish the quasi-traveling edge state along the smooth-curved edge which arises from a perturbation to a straight line. Similarly, edge states traveling along arbitrary slowly varying curves also could be extended by the same coordinate transformation shown in the arguments of Proposition \ref{prop.move}.

\section{Conclusion}

By defining the domain wall mass terms, we studied the topologically protected edge states via the Dirac equation with such generic masses. In this work, the traveling edge state tracking a straight line unidirectionally and keeps its shape along with the movement, which also is related to the chiral property. This peculiar feature of the explicit solution gives an insight into investigating the Dirac equation with more general smooth edges. The edge state moving along a varying edge will be very robust provided that the edge curvature is sufficiently small and there is negligible energy leaking into the bulk. To explain this subtle phenomenon, we introduced two typical edges which one is a large circle and the other is obtained by the small perturbation to a straight line. The asymptotic solution ansatz is derived by accepting the partial edge curve as the straight line and modulating the corresponding solution. Our rigorous study and numerical simulation demonstrated the edge states remain almost unchanged and highly concentrated on the slowly varying edge curves over a long time.

\section{Acknowledgements}

This work was partially supported by the National Natural Science Foundation of China (11871299). P.X. would acknowledge the support from Professor Hai Zhang and Department of Mathematics at HKUST.


\end{document}